\documentclass[12pt]{amsart}

\usepackage{amscd, amsfonts, amsmath, amssymb, amsthm, mathrsfs, appendix, enumerate, graphicx, flexisym, epsfig, float, graphicx, hyperref, pdfpages, tikz, mathtools, comment, tikz-cd}
\usepackage{multicol}
\usepackage[all,cmtip]{xy}
\usepackage{graphicx}
\usepackage{dynkin-diagrams}

\makeatletter
\@namedef{subjclassname@2020}{Mathematics Subject Classification \textup{2020}}

\newsavebox{\@brx}
\newcommand{\llangle}[1][]{\savebox{\@brx}{\(\m@th{#1\langle}\)}%
  \mathopen{\copy\@brx\kern-0.5\wd\@brx\usebox{\@brx}}}
\newcommand{\rrangle}[1][]{\savebox{\@brx}{\(\m@th{#1\rangle}\)}%
  \mathclose{\copy\@brx\kern-0.5\wd\@brx\usebox{\@brx}}}

\newtheorem{theorem}[equation]{Theorem}
\newtheorem{corollary}[equation]{Corollary}
\newtheorem{lemma}[equation]{Lemma}
\newtheorem{proposition}[equation]{Proposition}
\newtheorem{conjecture}[equation]{Conjecture}
\newtheorem{question}[equation]{Question}

\theoremstyle{definition}
\newtheorem{example}[equation]{Example}
\newtheorem{remark}[equation]{Remark}

\numberwithin{equation}{section}

\usepackage[all,cmtip]{xy}

\newcommand{\GL}{\mathrm{GL}}

\usepackage{tikz}

\usetikzlibrary{shapes,shadows,arrows,trees}
\usepackage{epstopdf}
\usepackage{inputenc}

\begin{document}

\title[Congruence subgroups]{Congruence subgroups of small Artin and Coxeter groups}

\author{Pravin Kumar}
\email{pravin444enaj@gmail.com}
\address{Beijing Institute of Mathematical Sciences and Applications, Beijing 101408, China.}

\subjclass[2020]{Primary 20F36; Secondary 20F55, 19B37} 
\keywords{Artin group, braid group, congruence subgroup problem, Coxeter group, integral Burau representation,  root system, Tits representation.}

\begin{abstract}

Small Coxeter groups are exactly those for which the Tits representation takes integral values, which makes the study of their congruence subgroups significant. In \cite{MR0938643}, Squier introduced a matrix representation of an Artin group defined over the ring $\mathbb Z[s^{\pm}, t^{\pm}]$ of Laurent polynomials in two variables. This representation simultaneously generalises the Tits representation of the associated Coxeter groups and the reduced Burau representation of braid groups. We define small Artin groups as those for which this representation becomes integral when evaluated at $s=1$ and $t=-1$. Consequently, the study of congruence subgroups of small Artin groups extends the classical notion of congruence subgroups of braid groups, which arise from the integral reduced Burau representation. In this paper, we examine Coxeter and Artin groups that possess the congruence subgroup property and identify several of their principal congruence subgroups at small levels.

\end{abstract}

\maketitle
\section{Introduction}
Congruence subgroups of a group are defined via a choice of an integral representation into $\mathrm{GL}(n, \mathbb{Z})$. The {\it level $m$ principal congruence subgroup} $G[m]$ of $G$ is then defined as the kernel of the composition map $G \rightarrow \mathrm{GL}(n, \mathbb{Z}) \rightarrow \mathrm{GL}(n, \mathbb{Z} / m)$. Congruence subgroups play an important role in the theory of arithmetic groups. We say that a group $G$ has the {\it congruence subgroup property} if every finite index subgroup of $G$ contains some principal congruence subgroup.
\par

 In the case of braid groups, the integral (unreduced) Burau representation has been used to define principal congruence subgroups. These subgroups have been studied in detail, and we refer to \cite{quotientbraid, MR4811535, MR4157115, MR3757477, MR3786423} for recent results and a survey of congruence subgroups of braid groups.  Among notable results, Arnold proved that the level two principal congruence subgroups of the braid group $B_n$ on $n$ strands is the pure braid group $PB_n$ \cite{MR0244266}.  Brendle and Margalit proved that the level four principal congruence subgroups of $B_n$ equals the kernel of the mod two abelianisation map $PB_n \to H_1(PB_n,\mathbb Z_2)$ \cite[Main Theorem]{MR3757477}. Further, Kordek and Margalit determined the first rational homology of the level four principal congruence subgroup of $B_n$ \cite{MR4462681}.
 \par
 
The congruence subgroups of Coxeter groups whose Tits representation is integral have been investigated recently in \cite{mypaper1}. It has been proved that infinite Coxeter groups with integral Tits representation which are not virtually abelian do not admit the congruence subgroup property \cite[Theorem 3.8]{mypaper1}. Further, it has been shown that the level four principal congruence subgroup of a right-angled Coxeter group identifies with its commutator subgroup \cite[Proposition 3.5]{mypaper1}. 
\par
  
Artin groups form a family of infinite discrete groups defined by presentations of a specific form, closely related to those of Coxeter groups. Despite their simple presentation, they remain not so well understood in many respects. The class of Artin groups is broad, encompassing braid groups, free groups, free abelian groups, as well as many more exotic examples.
\par

In \cite{MR0938643}, Squier introduced a matrix representation of an Artin group over the ring $\mathbb Z[s^{\pm}, t^{\pm}]$ of Laurent polynomials in two-variables. 
The representation is referred to as the generalised Burau representation, and generalises the Burau representation of the braid group (obtained by setting $s=1$),
 as well as the Tits representation of the corresponding Coxeter group (obtained by setting  $s =t= 1$). We say that an Artin group is a {\it small Artin group} if its generalised Burau representation is integral when we substitute $s=-t=1$. This makes the study of congruence subgroups of such groups relevant, which we explore in this paper.
\par

In this paper, we investigate the congruence subgroup property for Coxeter and Artin groups. The paper is organised as follows.   In Section \ref{section 2}, the basic preliminaries on Coxeter groups, Artin groups, and generalised Burau representation of Artin groups are presented.  In Section \ref{section 3}, we investigate the congruence subgroup problem for abstract groups and prove that a group with a finite index subgroup that surjects onto a free non-abelian group does not have the congruence subgroup property with respect to any integral representation (Proposition~\ref{thm:generalcsp}). In Section \ref{section 4}, we prove that a small Coxeter group which is virtually abelian admit the congruence subgroup property (Theorem~\ref{thm:cspaffine}).  In Section \ref{section 5}, we investigate the congruence subgroup property for Artin groups with respect to integral generalised Burau representations, and prove that Artin groups associated to Coxeter graphs whose connected components are not affine do not admit the congruence subgroup property (Theorem~\ref{mainthm:artin}). In Sections~\ref {section 6}~and~\ref{section 7}, we identify the principal congruence subgroups of Artin groups of small level. We prove that if $A$ is an Artin group associated with a spherical Coxeter graph, then $A/A[2] \cong W/Z(W)$, where $W$ is the corresponding Coxeter group (Theorem \ref{thm:level2}). We prove that if $A$ is a right-angled Artin group, then $A/A[4] \cong W/W^{'}$, where $W$ is the corresponding  right-angled Coxeter group (Theorem \ref{thm:level 4}).
\section{Preliminaries}\label{section 2}
\medskip

\subsection{Coxeter groups}
A group $W$  is called a {\it Coxeter group} if it admits a  presentation of the form $\big \langle S \mid R \big \rangle$, where $S= \{ w_i \mid i \in \Pi\}$ is the set of generators and 
\begin{equation}\label{Coxeter presentation}
R  = \big \{ (w_iw_j)^{m_{i,j}}~ \mid m_{i,j}\in \mathbb{N} \cup \{\infty\},~m_{i,j}=m_{j,i}, ~m_{i,j} = 1 ~\textrm{if and only if}~ i =j \big\}
\end{equation}
is the set of defining relations. The pair $(W, S)$ is called a  {\it Coxeter system}. If $m_{i,j} = \infty$, there is no relation between $w_i$ and $w_j$. We refer to $m_{i,j}$'s as {\it exponents} of the Coxeter system $(W, S)$. The cardinality of the set $S$ is called the {\it rank} of the Coxeter system, and we shall consider only finite rank Coxeter systems. Since Coxeter groups are completely determined by their exponents, we use matrices or graphs to encode this information.

 A {\it Coxeter matrix} on a (finite) set $S$ is a symmetric matrix $(m_{i,j})_{i,j\in S}$, where $m_{i,j} \in \mathbb{N} \cup\{\infty\}$ and satisfy $m_{i,i}=1$ for all $i\in S$ and $m_{i,j}=m_{j,i}\ge2$ for all $i,j\in S$ with $i\neq j$. Each such matrix can be represented by a graph, called {\it Coxeter graph} $\Gamma$, which is a labelled simple graph with $S$ as its set of vertices and two vertices $i,j \in S$ are joined by an edge if $m_{i,j}\ge3$, and such an edge is labelled with $m_{i,j}$ if $m_{i,j}\ge 4$. 

\par 

When we refer to an abstract group $W$ as a Coxeter group, we mean that $W$ has an associated Coxeter matrix or Coxeter graph and possesses a presentation of the form \eqref{Coxeter presentation}.  If the information is encoded in terms of the Coxeter matrix, we denote the Coxeter system by the pair $ (W, S) $, where $ S $ is the set of all Coxeter generators. If the information is encoded in terms of the Coxeter graph $ \Gamma $, we denote the Coxeter system as $W[\Gamma] $. Both notations are used interchangeably, depending on the context. 
\par
The Coxeter system $(W , S)$ is said to be {\it reducible} if $W = W_1 \times W_2$ , where $W_1 = \langle S_1 \rangle$ and $W_2 = \langle S_2\rangle $ for some subsets $S_1$ and $S_2$ of $S$. Otherwise, the Coxeter system is said to be {\it irreducible}. Equivalently, the Coxeter system is said to be irreducible if the Coxeter graph is connected.
\par

A homomorphism  $\psi:(W,S)\to (W',S')$  between Coxeter groups is called a {\it graph homomorphism} if $\psi(s)\in S'$ or $\psi(s)=1$ for all $s\in S$, and every $s'\in S'$ is of the form $\psi(s)$ for some $s\in S$. In particular, $\psi$ is surjective.
\par
Let $W$ be a Coxeter group given by a Coxeter presentation $W= \langle S \mid R \rangle $, where $S= \{ w_i \mid i \in \Pi\}$. Let $V$ be the real vector space spanned by the set $\{e_i \mid i \in \Pi \}$. Define a symmetric bilinear form $B$ on $V$ by
$$
B\left(e_{i}, e_{j}\right)= \begin{cases}-\cos \left(\frac{\pi}{m_{i, j}}\right) & \text { if } m_{i, j} \neq \infty, \\ -1 & \text { if } m_{i, j}=\infty. \end{cases}
$$
Then, for each $i \in \Pi$, the linear map $\overline{\rho}_{i}: V \rightarrow V$ given by 
\begin{equation}\label{Tits representation formula}
\overline{\rho}_{i}(v)=v-2 B\left(e_{i}, v\right) e_{i},
\end{equation}
 defines an automorphism of $V$. The following result is a folklore \cite[Chapter V, Section 4]{Bourbaki}.

\begin{theorem}
The map $\rho: W \rightarrow \GL(V)$ defined through $\rho\left(w_{i}\right)=\overline{\rho}_{i}$ is a faithful representation of $W$.
\end{theorem}

The  representation $\rho: W \rightarrow \GL(V)$ is called the {\it Tits representation} of $W$. It is easy to see that Tits representation of Coxeter group is integral if and only if all expontents $m_{i,j}\in\{1,2,3,\infty\}$. We say that a group is a {\it small Coxeter group} if it admits a Coxeter system such that each exponent $m_{i, j}$ is either $\infty$ or less than or equals to 3. For instance, symmetric groups, right-angled Coxeter groups and universal Coxeter groups are small Coxeter groups.
\par

The classification of finite irreducible Coxeter groups consists of four families of groups: $ A_n $ for $ n \geq 1 $, $ B_n $ for $ n \geq 2 $, $ D_n $ for $ n \geq 4 $, and $ I_2(p) $ for $ p \geq 5 $. Additionally, there are six exceptional groups: $ E_6 $, $ E_7 $, $ E_8 $, $ F_4 $, $ H_3 $, and $ H_4 $ (see the Figure~\ref{fincoxgraph}). We note that the labelling of vertices in Figure~\ref{fincoxgraph} is same as given in \cite{Bourbaki}.

  Every finite Coxeter group can be expressed as a direct product of a finite number of these irreducible groups.
\begin{figure}
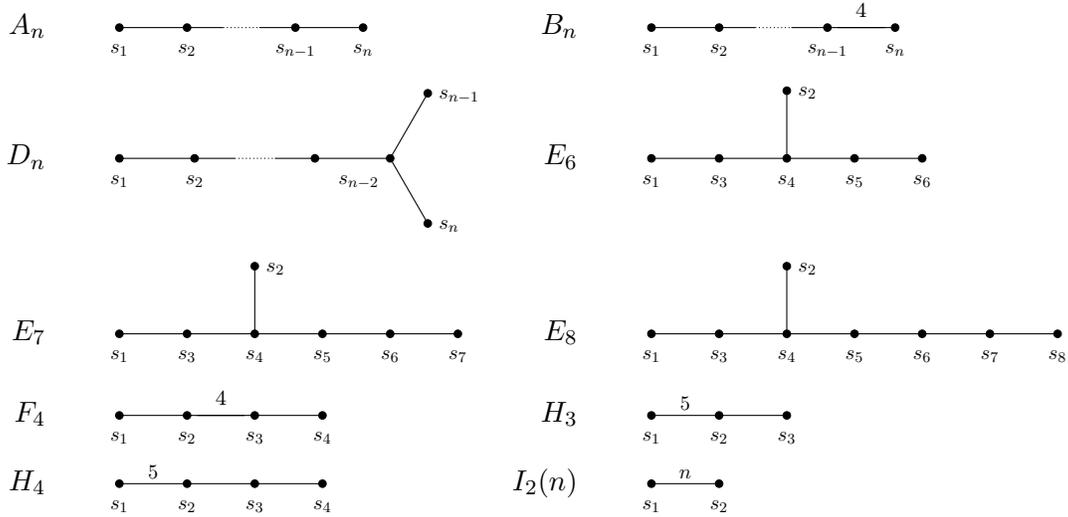

\begin{small}
$$
\begin{aligned}
A_n\quad\quad &\begin{dynkinDiagram}[labels={s_1,s_2,s_{n-1},s_n}, edge length=0.9cm]{A}{}
\end{dynkinDiagram}& 
B_n\quad\quad &\begin{dynkinDiagram}[labels={s_1,s_2,s_{n-1},s_n}, edge length=0.9cm]{A}{}
    \draw (root 3) -- node[above=0.3pt] {\tiny 4} (root 4);
\end{dynkinDiagram}&\\
D_n\quad\quad &\begin{dynkinDiagram}[labels={s_1, s_{2},  ,{\small ~~~s_{n-2}},s_{n-1}, s_{n}}, edge length=1.0cm]{D}{}
\end{dynkinDiagram}&
E_6 \quad\quad&\begin{dynkinDiagram}[labels={s_1,s_2,s_3,s_4,s_5,s_6}, edge length=0.9cm]{E}{6}
\end{dynkinDiagram}&\\
E_7 \quad\quad&\begin{dynkinDiagram}[labels={s_1,s_2,s_3,s_4,s_5,s_6,s_7}, edge length=0.9cm]{E}{7}
\end{dynkinDiagram}&
E_8 \quad\quad&\begin{dynkinDiagram}[labels={s_1,s_2,s_3,s_4,s_5,s_6,s_7,s_8}, edge length=0.9cm]{E}{8}
\end{dynkinDiagram}&\\
F_4\quad\quad& \begin{dynkinDiagram}[labels={s_1,s_2,s_3,s_4}, edge length=0.9cm]{A}{4}
    \draw (root 2) -- node[above=0.3pt] {\tiny 4} (root 3);
\end{dynkinDiagram}
&
H_3\quad\quad &\begin{dynkinDiagram}[labels={s_1,s_2,s_3},edge length=0.9cm]{H}{3}
\end{dynkinDiagram}
&\\
H_4\quad\quad & \begin{dynkinDiagram}[labels={s_1,s_2,s_3,s_4},edge length=0.9cm]{H}{4}
\end{dynkinDiagram}
&
I_2(n)\quad\quad &\dynkin [Coxeter,gonality=n, edge length=0.9cm, labels={s_1,s_2}]I{}&
\end{aligned}
$$
\end{small}
\caption{Coxeter graphs of irreducible finite Coxeter groups.}
\label{fincoxgraph}
\end{figure}

\par

\begin{remark}\label{rmk:finrig}
Let $(W,S)$ be an irreducible Coxeter system. The group $W$ can be expressed as a direct sum of two non-trivial Coxeter groups if and only if $W$ is not isomorphic to $W[\Gamma]$, where $\Gamma = I_2(4k+2)$ or $B_{2k+1}$ ($k \geq 1$) \cite[Theorem 3.3]{MR2240393}. For $W[I_2({4k+2})]$ and $W[B_{2k+1}]$, the following direct product decompositions holds.
 $$W[I_2(4k+2)] =W[I_2(2k+1)]\times W[A_1]=W[I_2(2k+1)] \times \mathbb Z_2$$
 and 
 $$W[B_{2k+1}]=W[D_{2k+1}]\times W[A_1].$$
 \end{remark}
 
 \begin{remark}
 From the classification of irreducible finite Coxeter groups, the Coxeter group of type $A_n$, $D_n$, and the exceptional groups $E_6$, $E_7$, and $E_8$ are small Coxeter groups. Additionally, the Coxeter group of type $I_6$ and $B_{2k+1}$ also have small Coxeter systems (cf. Remark \ref{rmk:finrig}). However, the remaining irreducible finite Coxeter groups: $B_{2n}=C_{2n}$ with $n \geq 1$, $I_2(n)$ with $n \geq 7$ or $n=5$, and the groups $F_4$, $H_3$, and $H_4$, do not have any small Coxeter systems due to their rigid properties as mentioned in Remark \ref{rmk:finrig}.
\end{remark}

\subsection{Artin groups}
A group $A$  is called an {\it Artin group} if there exist a  Coxeter system $(W, S)$ with Coxeter matrix $(m_{i,j})$ such that $A$ admits a  presentation of the form  
\begin{equation}\label{Artin presentation}
A = \big\langle S' \mid \underbrace{a_i a_j a_i \cdots}_{m_{i,j} \text{ terms}} = \underbrace{a_j a_i a_j \cdots}_{m_{i,j} \text{ terms}} \text{ for all } i \neq j \text{ with } m_{i,j} < \infty \big\rangle,
\end{equation}
where $S'= \{ a_i \mid i \in \Pi\}$ is a set in one-to-one correspondence with the set  $S$ of Coxeter generators. The pair $(A, S')$ is called an Artin system corresponding to the Coxeter system $(W,S)$. Given a Coxeter sytem $(W,S)$, there is a natural projection homomorphism $\pi:A\to W$ defined by sending each Artin generator $a_i\in S'$ to the corresponding Coxeter generator $w_i\in S$. The kernel of this homomorphism is called the {\it pure artin group} corresponding to the Coxeter system $(W,S)$. The homomorphism $\pi$ has a natural set-theoretic section $\iota: W \to A$ defined as follows. Let $w = s_1\ldots s_r \in W$ be any reduced expression of $w$ and we set $\iota(w)=a_1\ldots a_r \in A$. Note that $\iota$ is not a homomorphism. 

 If the information of exponents are encoded in the graph $\Gamma$, then we denote for Artin group corresponding to Coxeter graph $\Gamma$ as $A[\Gamma]$.  We say that $\Gamma$, $A[\Gamma]$ or $W[\Gamma]$ is {\it spherical} if $W[\Gamma]$ is a finite group, and it is {\it right-angled} if $m_{i,j} \in \{2,\infty\} $ for each $i\neq j$. Furthermore, it is called {\it crystallographic} if $m_{i,j}\in\{2,3,4,6,\infty\}$ for each $i\neq j$.

The {\it fundamental element} of spherical Artin group $A[\Gamma]$ is defined to be $\Delta= \iota(w_0)$, where $w_0$ denotes the element of $W[\Gamma]$ of maximal length. 
If $\Gamma$ to be connected and spherical, then the center $Z(A[\Gamma])$ of $A[\Gamma]$ is an infinite cyclic subgroup generated either by $\delta= \Delta^2$ if $\Gamma$ is $A_n$, $n\ge 2$, $D_{2n+1}$, $n\ge 2$,  $E_6$, and $I_2(2p+1)$, $p\ge 2$ or $\delta=  \Delta$ otherwise.

\subsection{Generalised Burau representation of Artin groups}
Let $A[\Gamma]$ be Artin group associated to Coxeter graph $\Gamma$ with Coxeter matrix $M=(m_{i,j})$. Let $\Lambda$ denote the Laurent-polynomial ring $\mathbb{R}\left[s, s^{-1}, t, t^{-1}\right]$, where $s$ and $t$ are indeterminates over $\mathbb{R}$. Define $K=K(M)$ to be the $n \times n$ matrix $\left[a_{i j}\right]$ over $\Lambda$, where
$$
a_{i j}= \begin{cases}-2 s \cos (\pi / m_{i, j}) & i<j, \\ 1+s t & i=j, \\ -2 t \cos (\pi / m_{i, j}) & i>j .\end{cases}
$$
Since $\operatorname{det}(K)\neq 0$, the form $\langle - , - \rangle$ is non-degenerate.  We introduce an analogue of complex conjugation in the Laurent-polynomial ring $\Lambda$ : if $x \in \mathbb{R}$ then $\bar{x}=x$, $\bar{s}=s^{-1}$ and $\bar{t}=t^{-1}$, extended to $\Lambda$ additively and multiplicatively. Note that if we substituted $s$ and $t$ with complex numbers of norm 1, then we recover ordinary complex conjugation. We extend the definition of conjugation to matrices entrywise and, if $A$ is a matrix over $\Lambda$, we define $A^*=\bar{A}^{\prime}$. For example, note that $K^*=s^{-1} t^{-1} K$.
\par
Let $V$ denote the free $\Lambda$-module with basis $\left\{e_1, \ldots, e_n\right\}$, and, as above, identify each $v \in V$ with its column vector of coordinates. For $u, v \in V$, we define $\langle u, v\rangle=u^* K v$. 
\begin{example}
For braid groups, we have
$$
\langle e_i, e_j \rangle = \begin{cases}
1+st & \text{ if } i=j, \\
-s & \text {if } j=i+1, \\
-t & \text {if } j=i-1, \\
0 & \text {if } |i-j| \geq 2.
\end{cases}
$$
\end{example}

For each $a_i \in \Pi$, the $\Lambda$-module homomorphism $\overline{\sigma}_{i}: V \rightarrow V$ given by 
\begin{equation}\label{generalised burau representation formula}
\overline{\sigma}_{i}(v)=v-\left\langle e_i, v\right\rangle e_i = v - (v^tKe_i)e_i,
\end{equation}
defines an automorphism of $V$. 
\begin{theorem}\cite[Theorem 1]{MR0938643}
The map $\sigma: A[\Gamma] \to \GL(V)$ defined through $\sigma\left(a_i\right)=\overline{\sigma}_i$ is a representation of $A[\Gamma]$.
\end{theorem}
The representation $\sigma$ provides a matrix representation of the Artin group $A[\Gamma]$ and this representation are faithful if the rank is two \cite{MR0938643}. The  representation $\sigma: A[\Gamma]\rightarrow \GL(V)$ is called the {\it generalised Burau representation} of the Artin group $A[\Gamma]$.
\par
 Note that Tits representation for Coxeter groups can be obtained from $K$ by substituting $s=t=1$. Thus, form $K$  can be seen as a deformation of the bilinear form associated with the Coxeter matrix $M = (m_{i j} )$ of the Coxeter group $W$.

\begin{remark}
 It is also easy to see that this representation factors through a one-parameter Hecke algebra associated with Coxeter groups, with the parameter $q=st$. Indeed, the following relations hold for all $i = 1,\ldots,n$:
\begin{equation}\label{eq:hecke}
\sigma(a_i)^2 + (q-1)\sigma(a_i)-q =0.
\end{equation}
\end{remark}
For an integer $n \ge 2$, the \textit{braid group} $B_n$ is defined as the group generated by the set $$S = \{\sigma_1, \sigma_2, \dots, \sigma_{n-1}\}$$ and satisfying the following defining relations:
$$\sigma_i\sigma_{i+1}\sigma_i = \sigma_{i+1}\sigma_i\sigma_{i+1},  \text{ and } \sigma_i\sigma_j = \sigma_j\sigma_i \text{ whenever } |i-j| \geq 2 .$$

\medskip

Let $I_m$ stands for the $m\times m$ identity matrix for all integer $m\geq 1$. The generalised Burau representation $\sigma: B_n \to \GL(n-1;\mathbb R)$ is given by $\sigma_i \mapsto Z_i$, where 
\begin{small}
\begin{equation*}
Z_i = 
\begin{cases}
\left(\begin{array}{cc|c}
-st & s & 0 \\
0 & 1 & \\
\hline 0 & 0 & I_{n-3}
\end{array}\right) & \text{if}\; i= 1,\\
\\
\left(\begin{array}{c|ccc|c}I_{i-2} & 0 & 0 & 0 & 0 \\ \hline & 1 & 0 & 0 &\\ 0 & t & -st & s & 0 \\ & 0 & 0 & 1 & \\ \hline 0 & 0 & 0 & 0 & I_{n-(i+2)}\end{array}\right) & \text{if}\;  2\leq i \leq n-2,\\
\\
\left(\begin{array}{c|c}
I_{n-3} & 0 \\
\hline
0 & \begin{array}{cc}
1 & 0 \\
t & -st
\end{array}
\end{array}\right) & \text{if}\; i= n-1.
\end{cases}
\end{equation*}
\end{small}
\begin{remark}
By evaluating $s=1$, we obtain the transpose of the reduced Burau representation as given in \cite[Theorem 3.9]{MR2435235}, as well as the reduced Burau representation as given in \cite[Lemma 3.11.1]{MR0375281}.  
\end{remark}
\medskip

\section{The congruence subgroup problem for abstract groups}\label{section 3}
Given a representation $\phi:G \to \mathrm{GL}(n, \mathbb{Z})$ of a group $G$ and an integer $m \ge 2$, one defines the {\it principal congruence subgroup} $G[m]$ of level $m$ as the kernel of the composition $$G \to \GL(n, \mathbb{Z}) \rightarrow \GL(n, \mathbb{Z}_m).$$ Notice that $G[m] \le G[k]$ for each divisor $k$ of $m$. A finite index subgroup of $G$ containing some principal congruence subgroup is called a {\it congruence subgroup}. We say that the group $G$ has the {\it congruence subgroup property}, with respect to $\phi$, if every finite index subgroup is a congruence subgroup. The {\it congruence subgroup problem} asks whether a group has the congruence subgroup property.

The definition of a principal congruence subgroup depends on the choice of representation. We say that two integer representations, $\phi_1: G \to \GL(n, \mathbb{Z})$ and $\phi_2:G \to \GL(n, \mathbb{Z})$ are equivalent if there exists a matrix $P\in \mathrm{GL}(n, \mathbb{Z})$ such that $\phi_2(g)=A^{-1}\phi_1(g)A$ for all $g \in G$. In this case, for any $g\in G$, $\phi_1(g)$ is the identity matrix modulo $m$ if and only if $\phi_2(g)$ is the identity matrix modulo $m$ for all $m\geq 2$. Therefore, the notion of a principal congruence subgroup of a group $G$ is well-defined up to this equivalence relation.

Consider the representation $\rho_1:G\to \GL(n_1, \mathbb Z)$ and $\rho_2: H \to \GL(n_2, \mathbb Z)$. We can define the direct sum representation $\rho_1\times \rho_2: G \times H \to \GL(n_1+n_2, \mathbb Z)$ as 
$$
(\rho_1 \times \rho_2)(g, h) = \left(\begin{matrix}
\rho_1(g) & 0 \\
0 & \rho_2(h)
\end{matrix}\right).
$$

It is straightforward to see that $(G\times H) [m] = G[m]\times H[m]$ for each $m\geq 2$.  Indeed,  if $(g, h) \in G \times H$, the matrix
$(\rho_1 \times \rho_2)(g, h) $ is congruent to the identity modulo  $m$ if and only if both matrix $\rho_1(g)$ and $\rho_2(h)$ are congruent to the identity modulo  $m$. Therefore, we can conclude that $(g,h) \in (G \times H)[m]$ if and only if $g \in G[m] \text{ and } h \in H[m]$.

\begin{proposition}\label{prop:cspdirect}
If $(G\times H, \rho_1\times \rho_2)$ admits the congruence subgroup property, then both $(G,\rho_1)$ and $(H,\rho_2)$ admit the  congruence subgroup property.
\end{proposition}
\begin{proof}
It is enough to show that $(G,\rho_1)$ has the congruence subgroup property. Let $K$ be a finite index subgroup of $G$. Then, $K\times H$ is a finite index subgroup of $G\times H$. Given that $G\times H$ has the congruence subgroup property, there exists an integer $m$ such that $G[m]\times H[m]=(G\times H)[m]\subset K\times H$. By applying the projection homomorphism onto $G$, we conclude that $G[m]\subset K$. This completes the proof of the proposition. 
\end{proof}

Given a group $G$, the {\it profinite topology} or the {\it Krull topology} on $G$ is generated by a sub-basis consisting of all finite index normal subgroups of $G$ and their left cosets \cite[Section 3.7]{surybook}. The profinite topology turns $G$ into a topological group, and we denote its completion by $\widehat{G}$. In addition, if $G$ is a group admitting an integral representation $\rho$, then we can define another topology called the {\it congruence topology} on $G$ by taking all principal congruence subgroups and their left cosets as a sub-basis. Let $\overline{G}$ be the completion of $G$ under the congruence topology. Both $\widehat{G}$ and $\overline{G}$ are profinite groups, and there is a surjective homomorphism $\widehat{G} \to \overline{G}$. The {\it congruence kernel} $C_{\rho}(G):=\ker(\widehat{G} \to \overline{G})$ is the kernel of this morphism. Then, the group $G$ admitting the congruence subgroup property with respect to $\rho$ is equivalent to the kernel $C_{\rho}(G)$ being trivial.

 Equivalently, the group $G$ has a congruence subgroup property with respect to $\rho$ if and only if the map  $\widehat{G}=\underleftarrow{\lim}G/U\to\underleftarrow{\lim}G/G[m]$ injective? where, $U$ ranges over all finite index normal subgroups of $G$, and $G[m]$ ranges over all principal congruence subgroup with respect to $\rho$. 

 It is easy to see that $\overline{G}$ can be seen as subgroup of $\GL(n,\hat{\mathbb Z})$, indeed, we have the following inclusions:
\begin{eqnarray*}
\underleftarrow{\lim}G/G[m] & \leq  &\underleftarrow{\lim}_{m\in\mathbb{N}} \GL(n,\mathbb Z_m)\\
 & \leq & \GL(n,\underleftarrow{\lim}_{m\in\mathbb{N}}\mathbb Z_m)=\GL(n,\hat{\mathbb Z}).
\end{eqnarray*}

Therefore, the congruence subgroup property is equivalent to the question: Is the congruence map  $\hat{G} \to \GL(n, \hat{\mathbb{Z}})$ injective?
\par

\begin{proposition}\label{thm:generalcsp}
Let $G$ be a group. If $G$ has a finite index subgroup that surjects onto a free non-abelian group, then the pair $(G,\rho)$ does not admit the congruence subgroup property for any representation $\rho:G \to \GL(n,\mathbb Z)$.
\end{proposition}

\begin{proof}
Let $N$ be a finite index subgroup of $G$ that surjects onto free non-abelian group of rank $k$. The profinite completion $\widehat{N}$ is a subgroup of $\widehat{G}$ that surjects onto the profinite completion of the free group $\widehat{F_k}$. Since $\widehat{F_k}$ is projective, there exists a subgroup $K$ of $\widehat{N}$, that is isomorphic to $\widehat{F_k}$. Thus, we have $\widehat{F_k}\cong K \subset \widehat{N} \subset \widehat{G}$, which does not admit any injective continuous homomorphism into $\GL(m, \hat{\mathbb Z})$ for any $m$ \cite[Proposition 3.8.3]{surybook}. This implies that $\widehat{G}$ does not admit an injective morphism into $\GL(n,\widehat{\mathbb Z})$. In particular, $(G,\rho)$ does not have congruence subgroup property. 
\end{proof}
\begin{corollary}
Let $G$ be a group. If $G$ is virtually free, then the pair $(G,\rho)$ does not admit the congruence subgroup property  for any representation $\rho:G \to \GL(n,\mathbb Z)$.
\end{corollary}
\medskip
\section{The congruence subgroup problem for affine Coxeter groups}\label{section 4}
The Tits representation of small Coxeter group is integral which makes the study of congruence subgroups of these groups relevant. In \cite{mypaper1}, the author, Naik, and Singh investigate which small Coxeter groups have the congruence subgroup property. If $\mathcal{P}$ denote a property of groups, then a group is said to be virtually $\mathcal{P}$ if it has a finite index subgroup with property $\mathcal{P}$. 

\begin{theorem}\cite[Theorem 3.8]{mypaper1}\label{p1:t3.8}
An infinite small Coxeter group which is not virtually abelian does not admit the congruence subgroup property with respect to the Tits representation of its corresponding small Coxeter system.
\end{theorem}

In this section, we will prove that an infinite small Coxeter group which is virtually abelian does have the congruence subgroup property. This will provide a complete answer to the congruence subgroup problem for small Coxeter groups with respect to Tits representation. In terms of Coxeter graph, a Coxeter group is virtually abelian if and only if each connected components of its Coxeter graph is either spherical or affine. 

The classification of irreducible affine Coxeter groups consists of four families of groups: $ \tilde{A}_n $ for $ n \geq 2 $, $ \tilde{B}_n $ for $ n \geq 3 $, $ \tilde{C}_n $ for $ n \geq 2$, and $ \tilde{D}_n $ for $ n \geq 4 $. Additionally, there are six groups: $\tilde{A}_1$, $ \tilde{E}_6 $, $ \tilde{E}_7 $, $ \tilde{E}_8 $, $ \tilde{F}_4 $, and $ \tilde{G}_2$.

 \begin{remark}
 Among the classification of irreducible affine Coxeter groups, the Coxeter group of types $\tilde{A}_n \ (n \geq 2)$,  $\tilde{D}_n \ (n \geq 4),$ $\tilde{E}_6, \tilde{E}_7,  \tilde{E}_8$ and $\tilde{A}_1$ are small Coxeter groups. Affine (Weyl) Coxeter groups are strongly rigid \cite[Main Theorem]{MR1760693} and hence other irreducible affine Coxeter groups $\tilde{B}_n,  \tilde{C}_n, \tilde{F}_4,$ and $\widetilde{G}_2$  are not small.
  \end{remark}

Let $\Gamma$ be a irreducible spherical crystallographic Coxeter graph with simple roots $\{\alpha_1, \ldots, \alpha_n\}$, and corresponding root system $\Phi[\Gamma]$. Let $\Phi[\Gamma]^\vee$ be its dual root system defined by 
$$
\Phi[\Gamma]^\vee=\{\alpha^\vee | \alpha \in \Phi[\Gamma]\}
$$
where the dual roots  $\alpha^\vee = \frac{2\alpha}{\langle \alpha, \alpha \rangle}. $

Let $Q(\Phi[\Gamma])=\bigoplus_{i=1}^n \mathbb{Z} \alpha_i$ (resp. $Q(\Phi[\Gamma]^\vee)=\bigoplus_{i=1}^n \mathbb{Z} \alpha_i^\vee$) denote the root lattice of $\Phi[\Gamma]$ (resp. its dual root lattice),  and let $P(\Phi[\Gamma])= \bigoplus_{i=1}^n \mathbb{Z} \omega_i$ denote the weight lattice of $\Phi[\Gamma]$, where $\omega_j$ is the $j$-th fundamental weight. It is well-known that if $\Gamma= A_n$, $D_n$, $E_6$, $E_7$ or $E_8$, then $Q(\Phi[\Gamma])=Q(\Phi[\Gamma]^\vee)$.

The finite Coxeter group $W[\Gamma] = \langle s_1, s_2, \dots, s_n \rangle$ acts on the Euclidean space $V$ spanned by the roots and the affine Coxeter group $W[\tilde{\Gamma}] = \langle s_0, s_1, \dots, s_n \rangle$
is generated by the reflections
$$
s_i = s_{\alpha_i} \quad (i=1, \dots, n), \quad \text{and} \quad s_0 = s_{\alpha_0,1}
$$
where $\alpha_0 := -\theta$ is the {\it affine simple root}, with $\theta$ the highest root in $\Phi[\Gamma]$ and $s_{\alpha_0,1}$ is the reflection across the affine hyperplane $H_{\alpha_0,1}=\{x\in V | \langle x,\alpha_0 \rangle=1\}$.

It is a fundamental fact that the affine Coxeter group decomposes as a semidirect product
$$
W[\tilde{\Gamma}] = Q \rtimes W[\Gamma],
$$
where $Q$ is the (free) abelian normal subgroup consisting of translations by vectors in the dual root lattice. In particular, the translation $t_{\alpha_i^\vee}$ by simple dual roots $\alpha_i^\vee$ ($i=1,\ldots, n$) forms a basis for $Q$.

Each translation by a simple dual roots is a conjugate of the product $s_0s_\theta$ by an element of $W[\Gamma]$. Indeed, for each simple root $\alpha_i$, there exists an element $w_i\in W[\Gamma]$ such that $w_i(\theta)=\alpha_i$, Now the translation by the simple dual root $\alpha_i^\vee$ can be expressed as $$t_{\alpha_i^\vee}=w_i s_0s_\theta w_i^{-1}=w_i s_0w_i^{-1}s_i .$$
\par

The element $s_\theta$ can be explicitly expressed as a product of the simple reflections $s_1, \dots, s_n$. For instance, when $\Gamma =A_n$ for $n\geq 2$, the simple roots are given by 
$$
p_i=e_{i+1}-e_i \quad \text { for } i=1,2,\ldots n
 $$
and the highest root is $\theta=e_{n+1}-e_1$. If we choose $w=s_1 s_2 \cdots s_{n-1}$, then $w\cdot p_n= \theta$ and hence $$s_\theta=s_1 s_2 \cdots s_{n-1} s_n s_{n-1} \cdots s_2 s_1.$$
 
Similarly, for $\Gamma = D_n$ with $n\geq 4$, the simple roots are given by 
$$
p_i=e_i-e_{i-1} \quad  \text { for } i = 1,\ldots n-1, \quad p_n = e_{n-1}+e_n,
$$
and the highest root is $\theta = e_{1} + e_{2}$. If we choose $w=s_{2}\cdots s_{n-3}s_{n-1}s_1\cdots s_{n-2}$ then $w\cdot p_n = \theta$ and hence 

$$ s_\theta=s_{2}\cdots s_{n-3}s_{n-1}s_1\cdots s_{n-2}  s_n s_{n-2}\cdots s_1 s_{n-1}s_{n-3}\cdots s_{2}.$$

Using similar computations or by utilizing sagemath, we can find the expression for $s_\theta$ in terms of simple reflection for the cases where $\Gamma =E_6, E_7$ and $E_8$ (see also \cite[Table 4]{MR1490241}).  The table below gives an expression of $s_\theta$ in terms of the simple reflections of irreducible affine small Coxeter groups.
 \begin{table}[H]\label{table:stheta}
  \begin{small}
  \begin{center}
  \begin{tabular}{|c|c|c|}
\hline
\textbf{Type} & \textbf{Highest root $ \theta $} & \textbf{Reflection $ s_\theta $ (as word in $ s_i $)} \\
\hline
$ A_n $ & $ \theta = \alpha_1 + \alpha_2 + \cdots + \alpha_n $ & $ s_\theta = s_1 s_2 \cdots s_n \cdots s_2 s_1 $ \\
\hline
$ D_n $ & $ \theta = \alpha_1 + \alpha_2 + 2\alpha_{3} + \cdots  + 2\alpha_{n-1} + \alpha_n $ 
& \begin{tabular}{c}$ s_\theta=ws_1w^{-1}$ where \\ $w=s_{2}\cdots s_{n-3}s_{n-1}s_1\cdots s_{n-2}$ \end{tabular} \\
\hline
$ E_6 $ & $ \theta = \alpha_1 + 2\alpha_2 + 2\alpha_3 + 3\alpha_4 + 2\alpha_5 + \alpha_6 $ 
&\begin{tabular}{c} $ s_\theta =ws_1w^{-1}$ where \\$w= s_2 s_4 s_5 s_6 s_3 s_4 s_5 s_2 s_4 s_3 $ \end{tabular}\\
\hline
$ E_7 $ & $ \theta = 2\alpha_1 + 2\alpha_2 + 3\alpha_3 + 4\alpha_4 + 3\alpha_5 + 2\alpha_6 + \alpha_7 $ 
& \begin{tabular}{c} $ s_\theta = ws_1w^{-1}$ where \\ $w= s_1s_3s_4s_5 s_6 s_7 s_2 s_4 s_5 s_6 s_3 s_4 s_5 s_2 s_4 s_3$\end{tabular} \\
\hline
$ E_8 $ & $ \theta = 2\alpha_1 + 3\alpha_2 + 4\alpha_3 + 6\alpha_4 + 5\alpha_5 + 4\alpha_6 + 3\alpha_7 + 2\alpha_8 $ 
& \begin{tabular}{c} $s_\theta =ws_1w^{-1}$ where \\ $w=  s_8 s_7 s_6 s_5 s_4 s_3 s_1 s_2 s_4 s_5 s_6 s_7 s_8$\\ $s_3 s_4 s_5 s_6 s_7s_2s_4s_5s_6s_3 s_4 s_5s_2s_4s_3 $ \end{tabular}\\
\hline
\end{tabular}
\medskip
\caption{Expression for $s_\theta$ in terms for simple reflections.}
\end{center}
\end{small}
\end{table}
\begin{remark}\label{rmk:center}
Among the irreducible affine small Coxeter groups, the corresponding spherical Coxeter group has a non-trivial center for the types $\tilde{D}_{2n}$ (where $ n \geq 2 $), $ \tilde{E}_7 $, and $ \tilde{E}_8.$

For the Coxeter group $W[D_{2n}]$ (where  $ n \geq 2 $), the element $(s_1s_2\cdots s_{2n})^{2n-1}$ is a non-trivial central element. This can be verified by computing the Tits representation of $W[D_{2n}]$ on this element and observing that it equals $-I_{2n}$.
 Similarly, non-trivial central element for the Coxeter groups $W[E_7]$ and $W[E_8]$ can be found using Sagemath.

 A simple matrix caculation shows that, under the Tits representation of the affine Coxeter group $W[\tilde{D}_{2n}]$ (for $ n \geq 2 $), $ W[\tilde{E}_7] $, and $ W[\tilde{E}_8]$, the matrix corresponding to the non-trivial central element of the corresponding spherical Coxeter group is congruent to the identity modulo $n$ if and only if $n=2$. In other words, for $m\geq 3$, this element does not lie in the level $m$ principal congruence subgroups. For instance, if $\rho_{\tilde{D}_{6}}$ is the Tits representation of $W[\tilde{D}_{6}]$, then 
$$
\rho_{\tilde{D}_{6}}((s_1\ldots s_6)^5)=\left(\begin{matrix}
1 & 0 & 0 & 0 & 0 & 0 & 0 \\
2 & -1 & 0 & 0 & 0 & 0 & 0 \\
4 & 0 & -1 & 0 & 0 & 0 & 0 \\
4 & 0 & 0 & -1 & 0 & 0 & 0 \\
4 & 0 & 0 & 0 & -1 & 0 & 0 \\
2 & 0 & 0 & 0 & 0 & -1 & 0 \\
2 & 0 & 0 & 0 & 0 & 0 & -1
\end{matrix}\right).
$$
Similarly, if $w_{E_*}$ is non-trivial central element of $W[E_*]$ and $\rho_{\tilde{E}_*}$ is Tits representation of $W[\tilde{E}_*]$ where $*=7$ or $8$, then 
$$
\rho_{\tilde{E}_7}(w_{E_7})=\left(\begin{matrix}
1 & 0 & 0 & 0 & 0 & 0 & 0 & 0 \\
4 & -1 & 0 & 0 & 0 & 0 & 0 & 0 \\
4 & 0 & -1 & 0 & 0 & 0 & 0 & 0 \\
6 & 0 & 0 & -1 & 0 & 0 & 0 & 0 \\
8 & 0 & 0 & 0 & -1 & 0 & 0 & 0 \\
6 & 0 & 0 & 0 & 0 & -1 & 0 & 0 \\
4 & 0 & 0 & 0 & 0 & 0 & -1 & 0 \\
2 & 0 & 0 & 0 & 0 & 0 & 0 & -1
\end{matrix}\right)
$$
and 
$$
\rho_{\tilde{E}_8}(w_{E_8})=\left(\begin{matrix}
1 & 0 & 0 & 0 & 0 & 0 & 0 & 0 & 0 \\
4 & -1 & 0 & 0 & 0 & 0 & 0 & 0 & 0 \\
6 & 0 & -1 & 0 & 0 & 0 & 0 & 0 & 0 \\
8 & 0 & 0 & -1 & 0 & 0 & 0 & 0 & 0 \\
12 & 0 & 0 & 0 & -1 & 0 & 0 & 0 & 0 \\
10 & 0 & 0 & 0 & 0 & -1 & 0 & 0 & 0 \\
8 & 0 & 0 & 0 & 0 & 0 & -1 & 0 & 0 \\
6 & 0 & 0 & 0 & 0 & 0 & 0 & -1 & 0 \\
4 & 0 & 0 & 0 & 0 & 0 & 0 & 0 & -1
\end{matrix}\right).
$$
\end{remark}

The classification of normal subgroups of finite and affine Coxeter groups is well-known \cite[Theorem 0.2]{MR1490241}. 
  \begin{proposition}\label{affnor}
 If $W[\tilde{\Gamma}]$ is a irreducible affine Coxeter group and $H$ is a normal subgroup of $W[\tilde{\Gamma}]$, then one of the following is true:
 \begin{enumerate}
\item $H$ is a $W[\Gamma]$-invariant subgroup of $Q$;
\item $H$ is an extension of a $W[\Gamma]$-invariant subgroup $L$ of $Q$ containing $2 Q$ by the center of $W[\Gamma]$;
\item There exists a graph homomorphism $\psi: W[\tilde{\Gamma}] \rightarrow W[\Sigma]$ to a finite Coxeter group $W[\Sigma]$, and a subgroup $Z$ of the centre of $W[\Sigma]$, such that $H$ is the kernel of the homomorphism $W[\tilde{\Gamma}] \rightarrow W[\Sigma] / Z$ induced by $\psi$.
\end{enumerate}
  \end{proposition}
\begin{remark}\label{rmk:images_jtits}
If $W$ is a small Coxeter group with Coxeter generators $s_1,\ldots, s_n$, then the under the Tits representation $\rho$, the image of each generator $s_j$  is a matrix with $-1$ in the $(j,j)$-th entry, $1$ in all other diagonal entries, arbitrary entries in the $j$-th row, and zero in all non-diagonal entries outside the $j$-th row.
That is, 
$$
\rho(s_j)=\left(\begin{array}{c|c|c}I_{j-1} & 0 & 0 \\
\hline
a_{j,1}\ \ldots \ a_{j,j-1} & -1 & a_{j,j+1} \ \ldots \ a_{j,n} \\
\hline
0 & 0 & I_{n-j}
\end{array}\right),
$$
where $a_{j,k}=\begin{cases}
1 & \text {if } m_{j,k}=3, \\
2 & \text {if } m_{j,k}=\infty, \\
0 & \text {if } m_{j,k}=2.
\end{cases}$
\end{remark}
\begin{proposition}
If $\tilde{\Gamma}$ is an irreducible affine small Coxeter graph with the corresponding spherical Coxeter graph $\Gamma$, then the principal congruence subgroup $W[\tilde{\Gamma}][m]$ $(m\geq 3)$ is a $W[\Gamma]$-invariant subgroup of the translation subgroup $Q$. When $m=2$, $W[\tilde{\Gamma}][2]$ is an extension of a $W[\Gamma]$-invariant subgroup $L$ of $Q$ containing $2Q$ by the center of $W[\Gamma]$.
\end{proposition}
\begin{proof}
 The Coxeter group of type $\tilde{A}_n \ (n \geq 2)$,  $\tilde{D}_n \ (n \geq 4),$ $\tilde{E}_6$, $\tilde{E}_7$,  $\tilde{E}_8$ and $\tilde{A}_1$ are the only irreducible affine small Coxeter group. By Remark \ref{rmk:center}, the non-trivial element (if it exists) of the center of the corresponding spherical Coxeter group  is in $W[\tilde{\Gamma}][m]$ if and only if  $m=2$. By Proposition \ref{affnor}, it is enough to show that $W[\tilde{\Gamma}][m]$ is not the kernel of homomorphism of the form $W[\tilde{\Gamma}]\xrightarrow{q} W[\Sigma] \to W[\Sigma]/Z(W[\Sigma])$, where $q$ is a graph homomorphism and $W[\Sigma]$ is finite Coxeter group.

 By definition, any graph homomorphism from an irreducible affine Coxeter graph to a spherical Coxeter graph should map a generator to the identity element or map some pairs of generator  to the same generator in $W[\Sigma]$. 

When $\tilde{\Gamma}$ is an irreducible affine Coxeter graphs, using the Remark~\ref{rmk:images_jtits},  the image of each generator $s_j$ under the Tits representation $\rho$ atleast one non-diagonal entry of  $j$-th row is equal to $1$. Thus, we have $$\rho(s_j)\neq I \mod (m).$$

 For $i\neq j$, the image of  $s_is_j$ under $\rho $ is a matrix with the $i$-th and $j$-th row possibly containing non-zero entries with $j$-th row has atleast one non-diagonal entry equal to $1$. Thus, we have $$\rho(s_is_j)\neq I \mod (m).$$ Thus, for all $m\geq 2$, $W[\tilde{\Gamma}][m]$ is not the kernel of any graph homomorphism. This completes the proof of the proposition.
\end{proof}
\begin{corollary}
If $W$ is a small Coxeter group which is virtually abelian, then the level $m$ principal congruence subgroups of $W$ is a free abelian group for $m\geq 3$. 
\end{corollary}
\begin{remark}\label{rmk:freeabe}
Any finite index subgroup of a free abelian group of rank $n$ is also a free abelian group of the same rank $n$.
\end{remark}

Let $\tilde{\Gamma}$ be an irreducible affine small Coxeter graph with the corresponding spherical Coxeter graph $\Gamma$. Every $W[\Gamma]$-invariant subgroup of $Q$ is an integral multiple of a lattice listed in \cite[Table 5]{MR1490241}. For instance, when $\Gamma=A_n$ $(n\neq 3)$,  $\Lambda_{k,d}=k(dQ+\mathbb Z(n+1)\omega_l)$ where $k\in \mathbb N$ and $d~|~n+1$ form the complete list of $W[\Gamma]$-invariant subgroup of $Q$. 

For each irreducible affine small Coxeter graphs $\tilde{\Gamma}$, using the Table~\ref{table:stheta}, a simple matrix computation shows that $$\rho_{\tilde{\Gamma}}(s_0s_\theta)^m = I_{n+1} \mod (m)$$ for all $m\geq 2$, and 
$$\rho_{\tilde{\Gamma}}(s_0s_\theta)^k \neq I_{n+1} \mod (m)$$ for all $k<m$,
 where $|V(\tilde{\Gamma})|=n+1$ and $\rho_{\tilde{\Gamma}}$ is the Tits representation of $W[\tilde{\Gamma}]$. Thus, 
$$
\rho_{\tilde{\Gamma}}(t^m_{\alpha_i^{\vee}})=\rho_{\tilde{\Gamma}}(w_i(s_0s_\theta)^mw_i^{-1}) = I_{n+1} \mod (m)
$$
 for all $m\geq 2$ and $\rho_{\tilde{\Gamma}}(t^k_{\alpha_i^{\vee}}) \neq I_{n+1} \mod (m)$ for all $k<m$.

 Since $W[\tilde{\Gamma}][m]$ is $W[\Gamma]$-invariant subgroup of $Q$, we have that $ W[\tilde{\Gamma}][m]$  lies between $m Q$ and $m P$, where $P$ is the weight lattice.

  \begin{remark}
  When $\Gamma=\tilde{A}_1$, for each \(m \geq 3\), we have 
\begin{equation}\label{csA1}
\begin{aligned}
W[\tilde{A}_1][m]  &=\left\{\begin{array}{ll}
 \langle (s_0s_1)^m \rangle & ~\text {if} ~m~\textrm{is odd},\\
\langle (s_0s_1)^{m/2}\rangle  & ~\text {if} ~m~\textrm{is even},
\end{array}\right.\\
&=\left\{\begin{array}{ll}
\Lambda_{m,1}  & ~\text {if} ~m~\textrm{is odd},\\
\Lambda_{\frac{m}{2},2}  & ~\text {if} ~m~\textrm{is even},
\end{array}\right.
\end{aligned}
\end{equation}
and $W[\tilde{A}_1]$ has the congruence subgroup property \cite[Proposition 3.20 and Corollary 3.21]{mypaper1}. 
\end{remark}
  
  \begin{theorem}\label{thm:cspaffine}
An infinite small Coxeter group which is virtually abelian admits the congruence subgroup property with respect to the Tits representation of its corresponding small Coxeter system.
\end{theorem}
\begin{proof}
A Coxeter group is virtually abelian if and only if each connected components of its Coxeter graph is either spherical or affine. Let $\Gamma$ be a Coxeter graph such that the associated Coxeter group $W[\Gamma]$ is virtually abelian. Let $\Gamma_1, \ldots, \Gamma_l$ be the connected components of $\Gamma$ which are of affine type, and let $\Gamma_{\mathrm{sph}}$ be the spherical part of $\Gamma$. Then the Coxeter group decomposes as:
$$
W[\Gamma] = W[\Gamma_{\mathrm{sph}}] \times W[\Gamma_1] \times \cdots \times W[\Gamma_l].
$$
Further, if $W[\Gamma]$ is small, then for each $m\geq 2$, 
$$
W[\Gamma][m] = W[\Gamma_1][m] \times \cdots \times W[\Gamma_l][m].
$$
 Let $\Delta[\Gamma_k]=\{\alpha_{1,k}^{\vee},\ldots, \alpha_{n_k, k}^{\vee}\}$ denote the set of simple coroots of $\Phi[\Gamma_k]$, then, the union  $\Delta=\cup_{i=1}^l \Delta[\Gamma_i]$ forms a set of simple coroots for $\Phi[\cup_{i=1}^l \Gamma_i]$. 
 
 Let $$
 \begin{aligned}
 \overline{\Delta}[\Gamma_k]&=\begin{cases}\{(l+1)\omega_l\}  & ~\text {if} ~\Gamma_k=A_{l},\\
 \{2\omega_1\}\cup \{4\omega_j~|~ j=1,\ldots, l\}  & ~\text {if} ~\Gamma_k=D_l, \text{ l is odd}\\
  \{2\omega_j~|~ j=1,\ldots, l\}  & ~\text {if} ~\Gamma_k=D_l, \text{ l is even}\\
 \{3\omega_j ~|~ j=1,\ldots, 6\}  & ~\text {if} ~\Gamma_k=E_6,\\
 \{2\omega_j ~|~ j=1,\ldots, 7\}  & ~\text {if} ~\Gamma_k=E_7,\\
 \end{cases}
 \end{aligned}
 $$
 where $\omega_j$ is the $j$-th fundamental weight of $\Gamma_k$. Set $\overline{\Delta}[\Gamma_k]=\emptyset$ if $\Gamma_k=E_8$. It is easy to see that $\overline{\Delta}[\Gamma_k]\subset Q[\Phi[\Gamma_k]]$ and let $\overline{\Delta}=\cup_{i=1}^l \overline{\Delta}[\Gamma_i]$.

 Now, let $H$ be an arbitrary finite index subgroup of $W[\Gamma]$. For each $\alpha\in \Delta\cup\overline{\Delta}$, choose $m_{\alpha}\in \mathbb N$ such that $$t^{m_{\alpha}}_{\alpha} \in H.$$

Define $m = \mathrm{lcm} \{ m_{\alpha} \; | \; \alpha \in \Delta\cup\overline{\Delta} \}$, then, it follows that 
 \begin{equation}\label{tmalpha}
 t^m_{\alpha} \in H.
 \end{equation}

Now, for each $1\leq i \leq l$, the level $m$ principal congruence subgroup of $W[\Gamma_i]$ is equal to $m$ times one of the lattices listed in \cite[Table 5]{MR1490241} for the graph $\Gamma_i$. By Equation~\eqref{tmalpha}, every invariant sublattice of $W[\Gamma_i]$ between $mQ$ and $mP$ is contained in $H\cap W[\Gamma_i]$. In particular, $W[\Gamma_i][m]\subset H\cap W[\Gamma_i]$. Thus,  we have
 $$
 W[\Gamma][m] = W[\Gamma_1][m] \times \cdots \times W[\Gamma_l][m]\subset H.
 $$
This completes the proof of the theorem.

\end{proof}

\section{Congruence subgroup property for Artin groups}\label{section 5}
Let $A[\Gamma]$ be the Artin group associated with a Coxeter graph $\Gamma$ and let $\sigma: A[\Gamma] \rightarrow \GL(V)$ be its generalised Burau representation. By evaluating this representation at $s=1$ and $t=-1$, we obtain another representation $\tilde{\sigma}:A[\Gamma]\to \GL(n,\mathbb R)$, where $n$ is the rank of corresponding Coxeter system. 

\begin{example}
If $\Gamma$ is graph on $n$ vertices without any edges, then $A[\Gamma]$ is a free-abelian group and the representation $\tilde{\sigma}$ is a trivial representation.
\end{example}

In the case of the braid group, the representation $\tilde{\sigma}$ is a symplectic representation (up to conjugacy). The kernel of this symplectic representation is called {\it braid Torelli group}. This normal subgroup is a non-trivial subgroup of braid groups for $n\geq 2$ and is an infinite index subgroup for $n\geq 3$. These subgroups have been well-studied (see \cite{MR3323579}).

 \begin{proposition}\label{prop:centerkernel}
Let $\Gamma$ be a graph on $n$ vertices with no vertex of degree zero. If $b \in Z(A[\Gamma])$, then $\tilde{\sigma}(b)$ is a block matrix with each block either $I_m$ or $-I_m$, where $I_m$ is the identity matrix of order $m$.  In particular, if each connected component of $\Gamma$ contains an odd number of vertices, then $Z(A[\Gamma])\subset \ker(\tilde{\sigma})$. 
 \end{proposition}
    
\begin{proof}
Let $A[\Gamma]$ is an Artin group with Artin generators $a_1,\ldots, a_n$.  Under $\tilde{\sigma} $, the image of each generator $a_j$  is a matrix with $1$ in the diagonal entries, arbitrary entries in the $j$-th row, and zero in all non-diagonal entries outside the $j$-th row. Since $\Gamma$ has no vertices of degree zero, the $j$-th row of $\tilde{\sigma}(a_j)$ contains atleast one non-zero non-diagonal entry.

 Let $b \in Z(A[\Gamma])$ and $1\leq i\leq n$.  By comparing the matrices $\tilde{\sigma}(a_i^{-1}b a_i))$ and $\tilde{\sigma}(b)$, along with the fact that $\Gamma$ has no vertices of degree zero, we can conclude that all non-diagonal entries of $\tilde{\sigma}(b)$ are $0$ and the $(i,i)$-th and $(j,j)$-th entries of $\tilde{\sigma}(b)$ are same if $m_{i,j}\neq2$.  Furthermore, if $\Gamma$ is connected, then the matrix $\tilde{\sigma}(b)$ is a scalar matrix. Since the determinant of $\tilde{\sigma}(b)$ is equal to $1$, it follows that the scalar is a real $n$-th root of unity. Therefore, $\tilde{\sigma}(b)$ is either $I_n$ or $-I_n$ when $n$ is even, and $\tilde{\sigma}(b)$ is $I_n$ when $n$ is odd. This completes the proof of the proposition.
  \end{proof}

\begin{remark}
If $\Gamma =F_4$, then $Z(A[F_4])=\langle \Delta \rangle$, but $\tilde{\sigma}(\Delta)=-I_4$.
\end{remark}
\begin{corollary}\label{cor:kerspherical}
Let $\Gamma$ be a graph on $n$ vertices with no vertices of degree zero. If the group $Z(A[\Gamma])$ is non-trivial, then the kernel of the representation $\tilde{\sigma}$ is a non-trivial normal subgroup of $A[\Gamma]$. If $\Gamma$ is a spherical Coxeter graph with at least two vertices, then the representation $\tilde{\sigma}$ is not faithful.
\end{corollary}
\begin{proof}
If $b\in Z(A[\Gamma])$, then by Proposition \ref{prop:centerkernel}, either $b$ or $b^2$ is in the kernel of the representation $\tilde{\sigma}$. If $\Gamma$ is a spherical Coxeter graph with at least two vertices, then it is easy to see that $\Gamma$ satisfies the hypothesis of Proposition~\ref{prop:centerkernel}. This completes the proof of the corollary.
\end{proof}
For convenience of terminology, we say that a group is a {\it small Artin group} if it admits a Coxeter system such that each exponent $m_{i, j}$ is either $\infty$ or less than or equal to 3. For example, braid groups and right-angled Artin groups are small Artin groups.  The following is an immediate observation for small Artin groups.

\begin{lemma}\label{tits rep integral small}
The generalised Burau representation of an Artin group evaluated at $s=1$ and $t=-1$ is integral if and only if it is a small Artin group.
\end{lemma}
\begin{proof}
Let $A[\Gamma]$ be a Artin group and $\sigma: A[\Gamma] \rightarrow \GL(V)$ given by $\sigma(a_i)=\overline{\sigma}_i$ be its generalised Burau representation. Clearly, the entries of the matrix of $\overline{\sigma}_i$ lie in $\{0, 1, -1, \pm\langle e_i,e_j\rangle \}$. It follows that $\pm\langle e_i,e_j\rangle = \pm 2  \cos(\frac{\pi}{m_{i,j}})$ is an integer if and only if  $m_{i,j}= 1,2,3$ or ${\infty}$.
\end{proof}
Thus, the matrix representation of a small Artin group is integral when we substitute $s=-t=1$ and in this case, we call $\tilde{\sigma}$ as {\it integral generalised Burau representation}. This makes the study of congruence subgroups of such groups relevant.
\begin{remark}
 Let $\Gamma$ and $\Omega$ be two spherical type Coxeter graphs. If $A[\Gamma]$ is isomorphic to $A[\Omega]$, then $\Gamma = \Omega$ \cite[Theorem 1.1]{MR2098788}. Futher, it can be possibly that a nonspherical type Artin group can be isomorphic to a spherical type Artin group. Thus, if $\Gamma$ is spherical, then $\Gamma$ admits a small spherical Coxeter system if and only if $\Gamma=A_n,B_{2k+1}, D_n, E_6, E_7, E_8$ and $I_6$.
 \end{remark}

\begin{remark}\label{rmk:direct}
 Let $\Gamma$ be a graph with two connected components $\Gamma_1$ and $\Gamma_2$. Then, we have $W[\Gamma]=W[\Gamma_1]\times W[\Gamma_2]$ and $A[\Gamma]=A[\Gamma_1]\times A[\Gamma_2]$. It is easy to see that the Tits representation of $W[\Gamma]$ is the direct sum of the Tits representations of $W[\Gamma_1]$ and $W[\Gamma_2]$. Similarly, the generalised Burau representation of $A[\Gamma]$ is the direct sum of the generalised Burau representations of $A[\Gamma_1]$ and $A[\Gamma_2]$.
\end{remark}

The braid group $B_3$ and affine Artin group of type $\tilde{A_1}$ are virtually free and hence they do not possess congruence subgroup property with respect to any integral representations. However, almost all Artin groups are not virtually-free groups. For instance, if $\Gamma$ is non-complete graph,  then $A[\Gamma]$ is not virtually-free groups. This is because $ A[\Gamma]$ contains a subgroup isomorphic to  $\mathbb{Z}^2$ as subgroup and for a virtually free group, every finitely generated subgroup of must also be virtually free and $\mathbb{Z}^2$ is not virtually free.

\begin{theorem}\label{mainthm:artin}
Let $\Gamma$ be a small Coxeter graph. If $\Gamma$ has a connected component that is not an affine Coxeter graph, then $A[\Gamma]$ does not admit the congruence subgroup property.
\end{theorem}
\begin{proof}
By Proposition \ref{prop:cspdirect} and Remark \ref{rmk:direct}, it is enough to show that for any irreducible non-affine Coxeter graph $\Gamma$, the group $A[\Gamma]$ do not have congruence subgroup property.

 If $\Gamma$ is spherical and if $A[\Gamma]$ has congruence subgroup property, then every finite index subgroup contains the kernel of $\tilde{\sigma}$ and by Corollary \ref{cor:kerspherical}, the intersection of all finite index subgroups is non-trivial subgroup, which contradicts the fact that $A[\Gamma]$ is a residually finite group  \cite{MR1942303, MR2004479, MR1888796}.
 Let us assume that $\Gamma$ is neither spherical nor Affine type.  In this case, the corresponding Coxeter group $W[\Gamma]$ is an infinite group which is not virtually abelian. By \cite[Theorem II]{MR2700693}, $W[\Gamma]$ has a finite index subgroup $N$ which surjects onto a free group of rank two. The subgroup $\pi^{-1}(N)$ is finite index in $A[\Gamma]$ which surjects onto a non-abelian free group. By Proposition \ref{thm:generalcsp}, it follows that $A[\Gamma]$ do not have congruence subgroup property. This completes the proof of the theorem.
\end{proof}
We note that the only cases excluded in Theorem~\ref{mainthm:artin} are the small Coxeter graphs whose connected components are all of affine types. Here, an irreducible Artin group $A[\Gamma]$ is said to be affine if $\Gamma$ is an irreducible affine Coxeter graphs.

\medskip
\section{Level two principal congruence subgroups}\label{section 6}
In view of Lemma \ref{tits rep integral small}, the  generalised Burau representation of a small Artin group, when evaluating  at $s=1$ and $t=-1$, is integral, and hence it is interesting to explore its (principal) congruence subgroups. Let $A[\Gamma]$ be a small Artin group with $|V(\Gamma)|=n$, where $n \ge 2$. For each  $m \ge 2$, let  $$\tilde{\sigma}_m: A[\Gamma] \to  \GL(n,\mathbb Z_m)$$ be the composition of $\tilde{\sigma}: A[\Gamma]\to  \GL(n,\mathbb Z)$ with the modulo $m$ reduction homomorphism $\GL(n,\mathbb Z) \to \GL(n,\mathbb Z_m)$. Let $A[\Gamma][m]:=\ker(\tilde{\sigma}_m)$ denotes the level $m$ principal congruence subgroup of $A[\Gamma]$. We note that the definition depends on our Coxeter graph $\Gamma$.

\begin{question}
For each $m\geq 2$, the subgroup $A[\Gamma][m]$ is finitely presented. In particular, it admits a finite generating set. Does there exists a finite generating set that can be interpreted in terms of the root system $\Phi[\Gamma]$?
\end{question}
Let $A[\Gamma]$ be a small Artin group with Artin generators $a_1,a_2,\ldots,a_{n}$, where $n \ge 2$. Fixing the ordered basis $\{e_1,e_2,\ldots,e_{n}\}$ for the real vector space $V$, by definition of $\tilde{\sigma}$, the Artin group $A[\Gamma]$ acts on a vector space $V$ with basis $\{e_1, \ldots, e_n\}$ and an bilinear form given by
$$
\langle e_i, e_j \rangle = \begin{cases}
 -2 \cos (\pi / m_{i, j}), & i<j, \\ 0, & i=j, \\ 2 \cos (\pi / m_{i, j}), & i>j .
\end{cases}
$$
where each generator $a_i$ acts on $V$ via
$$
\tilde{\sigma}(a_i)(e_j) = e_j - \langle e_i, e_j \rangle e_i.
$$

Our first observation is the following result.

\begin{proposition}\label{prop:level2}
Let $\Gamma$ be a small Coxeter graph. Then, for each integer $m \ge 2$, we have
$$\langle \langle a_i^m ~\mid~ 1 \le i \le n\rangle \rangle_{A[\Gamma]}~ \unlhd ~A[\Gamma][m]$$
In particular, if $m=2$, then the pure Artin group $PA[\Gamma]$ is contained in $A[\Gamma][2]$.
\end{proposition}
\begin{proof}
When $\Gamma$ is a small Coxeter graph, then for each $1 \le i, j \le n$, $\langle e_i, e_j\rangle$ is an integer. Using the definition of $\tilde{\sigma}$,  we have
$$
\tilde{\sigma}(a_i)^m(e_j)=e_j-m\langle e_i, e_j\rangle e_i. 
$$ 
Reduce modulo $m$, we get 
$$
\tilde{\sigma}(a_i^m)(e_j) = e_j \mod(m) \quad \text { for all } j=1,\ldots, n.
$$
 Thus, the matrix representation of $\tilde{\sigma}$ on $a_i^m$ is the identity matrix when reduced modulo $m$. This implies that the normal closure $\langle \langle a_i^m \mid i=1,\ldots, n \rangle\rangle\subset A[\Gamma][m]$.

In particular, when $m = 2$,  the normal closure of $a_1^2, \ldots, a_n^2$ in $A[\Gamma]$ is the pure artin group $PA[\Gamma]$. Thus, we have $PA[\Gamma] \subset A[\Gamma][2]$.
\end{proof}
\begin{remark}
Under the hypothesis of Theorem~\ref{mainthm:artin}, for some integer $m$, the containment in Proposition~\ref{prop:level2} is a strict containment. Indeed, if $A[\Gamma][m]=\langle \langle a_i^m ~\mid~ 1 \le i \le n\rangle\rangle$ for all $m$, then for any finite index subgroup $H$ of $A[\Gamma]$  and for each $i$, there exist $n_i$ such that $a_i^{n_i} \in H$. Set $n=\prod n_i$. Then $a_i^n \in H$ for each $i$, and $\langle \langle a_i^n ~\mid~ 1 \le i \le n\rangle\rangle$  is a principal congruence subgroup contained in $H$. This contradicts the fact that $A[\Gamma]$ do not have the congruence subgroup property.
\end{remark}

Thus,  we have following commutative diagram:
\begin{equation}\label{phieqn}
\begin{tikzcd}
	{A[\Gamma]} & {{A[\Gamma]/A[\Gamma][2]\subset \GL(n,\mathbb Z_2)}} \\
	{} & {W[\Gamma]}
	\arrow["{\tilde{\sigma}_2}", from=1-1, to=1-2]
	\arrow["\pi"', from=1-1, to=2-2]
	\arrow["\phi"', from=2-2, to=1-2]
\end{tikzcd}
\end{equation}
\begin{lemma}\label{lemma:nonabelian}
If $\Gamma$ is not a right-angled small Coxeter graph, then the quotient $A[\Gamma]/A[\Gamma][2]$ is non-abelian. In particular, $A[\Gamma][2] \neq A[\Gamma]$.
\end{lemma}
\begin{proof}
Without loss of generality, we can assume that $m_{1,2}=3$. It is sufficient to show that $\tilde{\sigma}(a_1a_2)(e_2)\neq \tilde{\sigma}(a_2a_1)(e_2) \mod(2)$. Indeed, we have: 
\begin{align*}
\tilde{\sigma}(a_1a_2)(e_2)&=\tilde{\sigma}(a_1)(e_2)=e_2 + e_1\\
\tilde{\sigma}(a_2a_1)(e_2)&=\tilde{\sigma}(a_2)(e_2+e_1)=e_2+\tilde{\sigma}(a_2)(e_1)\\
&=e_2 + (e_1 - e_2) =e_1
\end{align*}
Thus, when reduced modulo $2$, we see that $\tilde{\sigma}_2(a_1a_2)\neq \tilde{\sigma}_2(a_2a_1)$. Consequently, the image of $a_1$ and $a_2$ in the quotient $A[\Gamma]/A[\Gamma][2]$ do not commutes. This completes the proof.
\end{proof}
\begin{remark}\label{rmknormalDn}
The classification of normal subgroups of finite and affine Coxeter groups is well understood (see~\cite{MR1490241}). In order to understand the normal subgroups of Coxeter group of type~$D_n $ for $n \geq 4$, we consider the following group homomorphisms: 
\begin{enumerate}
\item $\psi:W[D_n] \to W[A_{n-1}]\cong S_n$ defined by $s_{i} \mapsto \tau_i=(i, i+1)$ for $ 1 \leq i \leq n-1$ and $s_n \mapsto \tau_{n-1}$, 
\item $\psi_0:W[D_4]\to S_3$  defined by $s_1, s_3, s_4 \mapsto \tau_1=(1,2)$ and $s_2 \mapsto \tau_2=(2,3)$,
\item $\psi_1:W[D_4]\to S_4$ defined by $s_1, s_3, \mapsto \tau_1=(1,2)$, $s_2 \mapsto \tau_2=(2,3)$ and $s_4 \mapsto \tau_3=(3,4)$,
\item $\psi_2:W[D_4]\to S_4$ defined by $s_1 \mapsto \tau_1=(1,2)$, $s_2 \mapsto \tau_2=(2,3)$ and  $s_3,s_4 \mapsto \tau_3=(3,4)$, and
\item $\psi_3:W[D_4]\to S_4$ defined by $s_1, s_4 \mapsto \tau_1=(1,2)$, $s_2 \mapsto \tau_2=(2,3)$ and  $s_3 \mapsto \tau_3=(3,4)$.
\end{enumerate}
 The non-trivial normal subgroups of the Coxeter group $W[D_n]$ are the commutator subgroup $W[D_n]'$, the kernel of the homomorphism $\psi$, and the center $Z(W[D_n])$, when $n$ is even.  Additionally, for $n=4$, the subgroup $\ker(\psi_0)$, $\ker(\psi_1)$, $\ker(\psi_2)$, and $\ker(\psi_3)$ are also normal subgroup of $W[D_4]$. By \cite{MR1490241}, these are the only non-trivial normal subgroups of $W[D_n]$. Therefore, if $N$ be a normal subgroup of $W[D_n]$, then 
 \begin{enumerate}[(i)]
\item for $n \geq 5$, if $s_{n-1}s_{n} \not \in N$, then $N$ must be either $\{1\}$ or $Z(W[D_n])$. In particular, if $n$ is odd, then $N=\{1\}$. 
\item for  $n=4$, if $s_1s_3, s_3s_4,$ and $s_1s_4$ are not elements of $N$, then $N$ must be either $\{1\}$ or $Z(W[D_4])$.
\end{enumerate}
\end{remark}

\medskip

\begin{theorem}\label{thm:level2}
If $\Gamma$ is a spherical Coxeter graph, then $A[\Gamma][2]=\ker(A[\Gamma]\to W[\Gamma]/Z(W[\Gamma]))$.
\end{theorem}
\begin{proof}
By Proposition~\ref{prop:level2} and Lemma~\ref{lemma:nonabelian}, the group $A[\Gamma]/A[\Gamma][2]$ is a finite non-abelian quotient of $W[\Gamma]$.

 When $\Gamma$ is of type $A_n$ or $E_6$, the smallest non-abelian quotient of $W[\Gamma]$ is $W[\Gamma]$ itself. Thus, we have $A[\Gamma][2]=PA[\Gamma]$ when $\Gamma=A_n$ or $E_6$. 

When $\Gamma=E_7$ or $E_8$, the non-abelian quotient of $W[\Gamma]$ is either $W[\Gamma]$ or $W[\Gamma]/Z(W[\Gamma])$. By Corollary \ref{cor:kerspherical}, a lift (the Garside element) of the longest element of $W[\Gamma]$  lies in the kernel of $\tilde{\sigma}$.  Thus, we have $A[\Gamma]/A[\Gamma][2]\cong W[\Gamma]/Z(W[\Gamma])$. This shows that $A[\Gamma][2]=\ker(A[\Gamma]\to W[\Gamma]/Z(W[\Gamma])$ when $\Gamma=E_7$ or $E_8$.

 When $\Gamma=D_n$, we have the following cases based on the parity of $n$:
\par
Case 1: When $n$ is odd.
The non-abelian quotient of $W[D_n]$ is $W[D_n]$ or $W[A_{n-1}]$. In the latter case, the quotient map is $\psi:W[D_n] \to W[{A_{n-1}}]\cong S_n$ defined by $s_{i} \mapsto \tau_i=(i, i+1)$ for $ 1 \leq i \leq n-1$ and $s_n \mapsto \tau_{n-1}$. A simple matrix computation shows that $\phi(s_{n-1}s_{n}) \neq I_n \mod 2$, where $\phi$ is given in \eqref{phieqn}. Thus, $A[\Gamma][2]=PA[\Gamma]$.
\par
Case 2: When $n$ is even.
The non-abelian quotient of $W[D_n]$ is $W[D_n]$ or $W[D_n]/Z(W[D_n])$ or $W[A_{n-1}]$. Additional, when $n=4$, the symmetric group $S_3$ is  also quotient of a $W[D_4]$. A simple matrix computation and by Corollary \ref{cor:kerspherical}, we can conclude that $A[D_n]/A[D_n][2]\cong W[D_n]/Z(W[D_n])$.  This completes the proof of the theorem.
\end{proof}

The map $\phi:W[\Gamma]\to \GL(n, \mathbb Z_2)$ given in \eqref{phieqn} agrees with the composition of the Tits representation of $W[\Gamma]$ with the modulo two reduction map. Thus, we have  $\ker(A[\Gamma]\to W[\Gamma]/W[\Gamma][2])\subset A[\Gamma][2]$.

\begin{conjecture}
If $\Gamma$ is a non-spherical small Coxeter graph, then $A[\Gamma][2]=\ker(A[\Gamma]\to W[\Gamma]/W[\Gamma][2])$.
\end{conjecture}

Quotients of congruence subgroups of braid groups have been studied recently in \cite{quotientbraid, MR4811535}.

\begin{proposition}
Let $\Gamma$ be a Coxeter graph and let $k$ be an odd integer. Then 
$$A[\Gamma][k]/A[\Gamma][2k]\cong A[\Gamma]/A[\Gamma][2].$$
\end{proposition}
\begin{proof}
Let $a_1, \ldots, a_n$ be the Artin generators.  By Proposition~\ref{prop:level2}, we have $a^l \in A[\Gamma][l]$ for all $l \geq 2$. Let $q:A[\Gamma]\to A[\Gamma]/A[\Gamma][2]$ be the quotient map.  When restricted to $A[\Gamma][k]$, the map $q$ is surjective because $q(a_i^k)=a_i \in A[\Gamma]/A[\Gamma][2]$.
 Since $k$ is odd, the kernel of this restricted homomorphism is $A[\Gamma][2]\cap A[\Gamma][k]=A[\Gamma][2k]$.Therefore, by the first isomorphism theorem, the proposition follows.
\end{proof}
\medskip

\section{Level four principal congruence subgroups}\label{section 7}
Let $A= \langle S \mid R \rangle$ be a small Artin group with $S=\{a_1,a_2,\ldots,a_{n}\}$, where $n \ge 2$. Fixing the ordered basis $\{e_1,e_2,\ldots,e_{n}\}$ for the real vector space $V$, by Lemma \ref{tits rep integral small}, evaluating at $s=1$ and $t=-1$, we obtain the generalised integral Burau representation as the matrix representation $\tilde{\sigma}:A \to \GL(n,\mathbb Z)$. Let $\Gamma$ be a right-angled Coxeter graph, by definition of $\tilde{\sigma}$, the Artin group $A[\Gamma]$ acts on a vector space $V$ with basis $\{e_1, \ldots, e_n\}$ and an bilinear form given by
$$
\langle e_i, e_j \rangle = \begin{cases}
-2 & \text {if } m_{i,j}=\infty \text{ and } i<j, \\
2 & \text {if } m_{i,j}=\infty \text{ and } i>j,\\
0 & \text {if } m_{i,j}=2 \text{ or } i=j.
\end{cases}
$$
where each generator $a_i$ acts on $V$ via
$$
\tilde{\sigma}(a_i)(e_j) = e_j - \langle e_i, e_j \rangle e_i.
$$

\begin{proposition}\label{prop:level4raag}
Let $A[\Gamma]$ be the right-angled Artin group associated with a right-angled Coxeter graph $\Gamma$. Then, for each $m \ge 1$, 
$$\langle \langle \left(a_i\right)^m ~\mid~ 1 \le i \le n\rangle \rangle_{A[\Gamma]}~ \unlhd ~A[\Gamma][2m].$$ 
In particular, if $m=1$, then $A[\Gamma][2]=A[\Gamma]$.
\end{proposition}

\begin{proof}
It is straightforward to observe that $$(\tilde{\sigma}(a_i^m))(e_j)=e_j-m\left\langle e_i, e_j\right\rangle e_i.$$
Reducing modulo $2m$, We have
$$
(\tilde{\sigma}(a_i^m))(e_j)=e_j \mod (2m)
$$
 Thus, the matrix representation of $\tilde{\sigma}$ on $a_i^m$ is the identity matrix when reduced modulo $2m$. This implies that the normal closure $\langle \langle a_i^m \mid i=1,\ldots, n \rangle\rangle\subset A[\Gamma][2m]$.

In particular, when $m = 1$, we have  $A[\Gamma][2]=A[\Gamma]$. This completes the proof of the proposition.
\end{proof}

If $\Gamma$ is a right-angled Coxeter graphs, then by taking $m=2$ in Proposition~\ref{prop:level4raag},  we have following commutative diagram:
$$\begin{tikzcd}
	{A[\Gamma]} & {{A[\Gamma]/A[\Gamma][4]\subset \GL(n,\mathbb Z_4)}} \\
	{} & {W[\Gamma]}
	\arrow["{\tilde{\sigma}_4}", from=1-1, to=1-2]
	\arrow["\pi"', from=1-1, to=2-2]
	\arrow["\phi"', from=2-2, to=1-2]
\end{tikzcd}
$$

 Let $(a_{i,j})$ denotes the matrix representing of $\tilde{\sigma}(a_k)$ for $1 \leq k \leq n$, then one can see that
\begin{equation}\label{small coxeter matrix}
a_{i,j}=\begin{cases}
\delta_{i,j} & \text{ if } i\neq k,\\
\alpha(k,j) & \text{ if } i = k,
\end{cases}
\end{equation}
where $\alpha(k,k)=1$ and  if $k < j$ then $$\alpha(k,j)=\begin{cases}
1 & \text{ if } m_{k,j} = 3, \\
0 & \text{ if } m_{k,j} = 2, \\
2 & \text{ if } m_{k,j} = \infty, \\
\end{cases} 
$$ and if $k>j$, then 
$$\alpha(k,j)=\begin{cases}
-1 & \text{ if } m_{k,j} = 3, \\
0 & \text{ if } m_{k,j} = 2, \\
-2 & \text{ if } m_{k,j} = \infty, \\
\end{cases} 
$$
Now, if $1 \le k\neq \ell \le n$ and $(c_{i,j})$ denotes the matrix representing of $\tilde{\sigma}(a_ka_l)$, then 
\begin{equation}\label{cij}
c_{i,j}=\begin{cases}
    \delta_{i,j} & \text{ if } i\neq k,\quad i\neq \ell,\\
       \alpha(\ell,j) & \text{ if }  i=\ell,\quad j\neq \ell,\\
        1 & \text{ if } i=\ell,\quad j=\ell,\\
        	\alpha(k,j)+ \alpha(\ell,j)\alpha(k,\ell)    & \text{ if } i=k, \quad j\neq k, \ell\\
    \alpha(k,\ell) & \text{ if } i=k,\quad j=\ell,\\
    	1-\alpha(k,\ell)^2    & \text{ if } i=k, \quad j= k,
\end{cases}
\end{equation}
and if $1 \le k\neq \ell \le n$ and $(c_{i,j}')$ denotes the matrix representing of $\tilde{\sigma}(a_k^{-1}a_l^{-1})$, then 
\begin{equation}\label{cij'}
c_{i,j}'=\begin{cases}
    \delta_{i,j} & \text{ if } i\neq k,\quad i\neq \ell,\\
    -\alpha(\ell,j) & \text{ if }  i=\ell,\quad j\neq \ell,\\
        1 & \text{ if } i=\ell,\quad j=\ell,\\
        	-\alpha(k,j)+ \alpha(\ell,j)\alpha(k,\ell)    & \text{ if } i=k, \quad j\neq k, \ell,\\
    -\alpha(k,\ell) & \text{ if } i=k,\quad j=\ell,\\
		1-\alpha(k,\ell)^2    & \text{ if } i=k, \quad j= k,
\end{cases}
\end{equation}
where $\alpha(k,j)$ is defined in equation \eqref{small coxeter matrix}. Further, note that
$$c_{k,k}=c_{k,k}'=1+\alpha(\ell,k)\alpha(k,\ell)=1-\alpha(k,\ell)^2,\quad c_{\ell,\ell}=c_{\ell,\ell}'=1,\quad \textrm{and} \quad c_{k,\ell}=\alpha(k,\ell)=-\alpha(\ell,k)=- c_{\ell,k}.$$

For $j\neq \ell$, we set $$\gamma_j :=\begin{cases}
        1-\alpha(k,\ell)^2  & \text{ if } j= k,  \\
       	\alpha(k,j)- \alpha(\ell,j)\alpha(k,\ell) & \text{ if } j\neq k.
    \end{cases}$$

\begin{lemma}\label{prod matrix square}
Let $1 \le k\neq \ell \le n$  and $(d_{i,j})$ be the matrix representing of $\tilde{\sigma}(a_ka_la_k^{-1}a_l^{-1})$. Then 
$$d_{i,j}=\delta_{i,j} \quad \textrm{if} \quad i\neq k,\ell, $$
$$
d_{k,j}=\begin{cases}
      \alpha(k,\ell)^4-\alpha(k,\ell)^2+1 & \text{ if } j=k,\\
      \alpha(k,\ell)^3 & \text{ if }  j=\ell,\\
	\gamma_j\alpha(k,\ell)^2+\alpha(k,\ell)\alpha(\ell,j)  & \text{ if }  j\neq k ,\ell,
\end{cases}
\quad \textrm{and} \quad 
d_{\ell,j}=\begin{cases}
        \alpha(k,\ell)^3  &\text { if } j=k,\\
		\alpha(k,\ell)^2 + 1  &\text { if } j=\ell, \\
		\gamma_j\alpha(k,\ell)	 &\text { if } j\neq k, \ell.
      \end{cases}
$$

\end{lemma}

\begin{proof}
If $i\neq k, \ell$, then 
$$    d_{i,j} =\sum_{p=1}^{n} c_{i,p}c_{p,j}'   =\sum_{p=1}^{n} \delta_{i,p}c_{p,j}' =c_{i,j}' = \delta_{i,j}.$$ 
If $i=k$, then 
$$
\begin{aligned}
  d_{k,j}  &=\sum_{p=1}^{n} c_{k,p}c_{p,j}' \\
      &=\sum_{\substack{p=1\\p\neq k,\ell}}^{n}c_{k,p}c_{p,j}' + c_{k,k}c_{k,j}' + c_{k,\ell}c_{\ell,j}'\\
      &=\sum_{\substack{p=1\\p\neq k,\ell}}^{n}c_{k,p}\delta_{p,j} + c_{k,k}c_{k,j}' + c_{k,\ell}c_{\ell,j}'\\
      &=\begin{cases}
          c_{k,k}c_{k,k}' + c_{k,\ell}c_{\ell,k}' &\text { if } j=k,\\
          c_{k,k}c_{k,\ell}' + c_{k,\ell} &\text { if } j=\ell, \\
          c_{k,j}+c_{k,k}c_{k,j}'+c_{k,\ell}c_{\ell,j}' &\text { if } j\neq k, \ell,
      \end{cases} \\
      &=\begin{cases}
           \alpha(k,\ell)^4-\alpha(k,\ell)^2+1 & \text{ if } j=k,\\
      \alpha(k,\ell)^3 & \text{ if }  j=\ell,\\
	\gamma_j\alpha(k,\ell)^2+\alpha(k,\ell)\alpha(\ell,j)  & \text{ if }  j\neq k ,\ell.
\end{cases}
\end{aligned}
$$
If $i=\ell$, then 
$$
\begin{aligned}
  d_{\ell,j}  &=\sum_{p=1}^{n} c_{\ell,p}c_{p,j}' \\
      &=\sum_{\substack{p=1\\p\neq k,\ell}}^{n}c_{\ell,p}\delta_{p,j} + c_{\ell,k}c_{k,j}' + c_{\ell,\ell}c_{\ell,j}'\\
      &=\begin{cases}
         c_{\ell,k}c_{k,k}' + c_{\ell,k}'  &\text { if } j=k,\\
	c_{\ell,k}c_{k,\ell}'+ 1  &\text { if } j=\ell, \\
         c_{\ell,j} + c_{\ell,k}c_{k,j}' +c_{\ell,j}' &\text { if } j\neq k, \ell,
      \end{cases} \\
      &=\begin{cases}
        \alpha(k,\ell)^3  &\text { if } j=k,\\
		\alpha(k,\ell)^2 + 1  &\text { if } j=\ell, \\
		\gamma_j\alpha(k,\ell)	 &\text { if } j\neq k, \ell.
      \end{cases}
      \end{aligned}
$$
\end{proof}
If $A$  is right-angled, then
\begin{equation}\label{econ}
\alpha(i,j)=\begin{cases}
~~0\mod 2 &\text{if } i\neq j,\\
1\mod 2&\text{if } i = j,
    \end{cases} \quad , \quad \gamma_j=\begin{cases} 1\mod 4 & \text{if } j=k,\\
    ~~0 \mod 2 &\text{if } j\neq k.
    \end{cases} \quad \textrm{and} \quad d_{i,j} = \delta_{i,j} \mod 4.
\end{equation}

We denote the commutator subgroup of a group $G$ by $G^{'}$.

\begin{theorem}\label{thm:level 4}
Let $\Gamma$ be a right-angled Coxeter graph with no vertices of degree zero. Then  $A[\Gamma][4]= \ker (A[\Gamma] \to W[\Gamma]/W[\Gamma]')$, where $W[\Gamma]$ is the corresponding right-angled Coxeter group.
\end{theorem}
\begin{proof}
 Let $a_1, \ldots, a_n$ be Artin generators of right-angled Artin group $A[\Gamma]$. Let $q:W[\Gamma]\to \mathbb Z_2^{n}$ the abelianisation map of right-angled Coxeter group $W[\Gamma]$ and $\pi:A[\Gamma]\to W[\Gamma]$ be the natural quotient map.  Let $u_1,u_2,\ldots,u_{n}$ be generators for $\mathbb Z_2^{n}$, where $u_i=q(\pi(a_i))$.  Define a map $\psi:\mathbb Z_2^n \to \GL(n,\mathbb Z_4)$, given by $\psi(u_i)=\tilde{\sigma}(a_i) \mod 4$.  Equation \eqref{econ} implies that the matrix $(d_{i,j})$ is identity modulo 4.  Thus, the map $\psi$ is a group homomorphism. We claim that $\psi$ is injective. Let $w=u_{j_1}u_{j_2}\cdots u_{j_r} \in  \mathbb Z_2^{n}$ be a word of length $r \ge 1$, where $j_1<j_2<\cdots < j_r$. Using \eqref{cij} and induction on $r$, one can see that if $\ell>j_r$, then the $\ell$-th row of $\psi(w)$ has $1$ in the $(\ell,\ell)$-entry and $0$ in all other entries. Now, suppose that  $u=u_{i_1}u_{i_2}\cdots u_{i_k} \in \ker(\psi)$ for some $k\ge1$. Since $u_i \not\in \ker(\psi)$ for all $i$, we have $k>1$ and we can assume without loss of generality that $i_1<i_2<\cdots < i_k$ Note that the $i_k$-th row of $\psi(u_{i_1}u_{i_2}\cdots u_{i_{k-1}})$ has $1$ in the $(i_k,i_k)$-entry and all other entries are 0. Under our assumption on the graph $\Gamma$, we have that  $\psi(u_{i_k})$  is not equal to the identity matrix, i.e., there exists some $j$ such that the $(i_k, j)$-entry of $\psi(u_{i_k})$ is either $2$ or $-2$, $j$-th column of $\psi(u_{i_k})$ has $1$ in the $(j,j)$-entry, and the $(i_k, j)$-entry is $2$ or $-2$, while all other entries are 0. Thus, $\psi(u)$ has a non-zero $(i_k,j)$-entry equal to $2$ or $-2$, which is a contradiction. Hence, $\psi$ is injective, and we obtain $A[\Gamma][4]=\ker(q\circ \pi)$. This completes the proof of the theorem.
\end{proof}

\medskip

\noindent\textbf{Acknowledgments.}
The author would like to sincerely thank Prof. Mahender singh for helpful insights and careful reading of the paper. The author thanks BIMSA for the institute postdoctoral fellowship.
\medskip

\noindent\textbf{Declaration.}
The authors declare that there is no data associated to this paper and that there are no conflicts of interests.
\medskip

\bibliographystyle{plain}
\bibliography{template}

\end{document}